\documentclass[a4paper,reqno]{amsart}
\usepackage{amscd}
\usepackage{amssymb,latexsym}
\usepackage[centertags]{amsmath}
\usepackage{amsfonts}
\usepackage{amssymb}
\usepackage{amsthm}
\usepackage{newlfont}
\usepackage[all]{xy}
\usepackage{amsmath,amssymb}
\usepackage[latin1]{inputenc}
\usepackage{colortbl}
\usepackage[english]{babel}
\usepackage[normalem]{ulem}
\usepackage{graphicx}
\usepackage{color}
\usepackage{times}
\usepackage{verbatim}

\def\subsection#1{\removelastskip\par\medskip\setcounter{equation}{0}
\refstepcounter{subsection}\noindent{\bf\thesubsection. #1}}

\def\g{ $\displaystyle {\frak g}$ }
\def\g01{ $\displaystyle {\frak g}={\frak g}_{\bar 0}\oplus{\frak g}_{\bar
1}$ }
\def\g0{ $\displaystyle {\frak g}_{\bar 0}$ }
\def\g1{ $\displaystyle {\frak g}_{\bar 1}$ }

\newtheorem{teo}{Theorem}[section]
\newtheorem{cor}[teo]{Corollary}
\newtheorem{lem}[teo]{Lemma}
\newtheorem{exa}[teo]{Example}
\newtheorem{pro}[teo]{Proposition}
\theoremstyle{definition}
\newtheorem{defi}[teo]{Definition}
\theoremstyle{remark}
\newtheorem{rmk}[teo]{Remark}

\def\3{{\sf 3}}

\def\g{{\gamma}}

\begin{document}

\title[Quasicrossed product on $G$-graded quasialgebras]{Quasicrossed product on
$G$-graded quasialgebras}

\author[H. Albuquerque ]{Helena~Albuquerque}

\address{Helena~Albuquerque, CMUC, Departamento de Matem\'atica, Universidade de Coimbra, Apartado 3008,
3001-454 Coimbra, Portugal {\em E-mail address}: {\tt lena@mat.uc.pt}}

\thanks{The first and the second authors acknowledge financial assistance by  the Centro de Matem\'{a}tica da Universidade de Coimbra (CMUC), funded by the European Regional Development Fund through the program COMPETE
and by the Portuguese Government through the FCT - Funda\c{c}\~{a}o para a Ci\^{e}ncia e a Tecnologia
under the project PEst-C/MAT/UI0324/2011. Third author acknowledges the University of M\'alaga for the
grant to performs the research stay and is supported by the PAI's
with project number FQM-298, by the project of the Junta de
Andaluc\'{i}a FQM-7156 and by the scholarship of the University of
C\'adiz with reference PU2012-036-FPI}

\author[E. Barreiro]{Elisabete~Barreiro}

\address{Elisabete~Barreiro, CMUC, Departamento de Matem\'atica, Universidade de Coimbra, Apartado 3008,
3001-454 Coimbra, Portugal
{\em E-mail address}: {\tt mefb@mat.uc.pt}}{}


\author[Jos\'{e} M. S\'{a}nchez-Delgado]{Jos\'{e} M. S\'{a}nchez-Delgado}

\address{Jos\'{e} M. S\'{a}nchez-Delgado, Departamento de Matem\'{a}ticas, Universidad de C\'{a}diz, 11510 Puerto Real, C\'{a}diz, Spain {\em E-mail address}: {\tt txema.sanchez@uma.es}}


\begin{abstract}
The notion of quasicrossed product is introduced in the setting of $G$-graded quasialgebras, i.e., algebras endowed with a grading by a group $G$, satisfying a ``quasiassociative'' law. The equivalence between quasicrossed products and quasicrossed systems is explored. It is presented the notion of graded-bimodules in order to study simple quasicrossed products. Deformed group algebras are stressed in particular.
\end{abstract}

\maketitle

\bigskip

\section{Introduction}

The $G$-graded quasialgebras were introduced by H. Albuquerque and S. Majid about a decade ago \cite{AM1999},  and during the last years have been addressed with some collaborators (see \cite{AEPI2008} and the references therein).  Inspired by the theory of graded rings and graded algebras \cite{Fin_Dim_Algebras, Kar1989, Kar1990, NO1982}, in the present paper we extend the concepts
of crossed products and crossed systems to the context of $G$-graded quasialgebras.
The division graded quasialgebras are quasicrossed products, as well as some notable nonassociative algebras such as the deformed group algebras, for example, Cayley algebras and Clifford algebras.
Among them, we stress the octonions with potential relevance to many interesting fields of mathematics, namely Clifford algebras and spinors, Bott periodicity, projective
and Lorentzian geometry, Jordan algebras, and the exceptional Lie groups. We also refer to its applications to physics, such as, the foundations of quantum mechanics and string theory.
There are some others interesting examples of quasicrossed products, such as the simple antiassociative superalgebras.  We prove some basic results about quasicrossed products and quasicrossed systems, emphasizing the case of deformed group algebras.
Our work extends the study on unital antiassociative quasialgebras with semisimple even part presented in \cite{AS2004}.

In Section 2, we collect the basic definitions and properties related to graded quasialgebras. Section 3 is devoted to a review of some properties of the set of the units of a graded quasialgebra. In Section 4 the quasicrossed products and systems are defined and interrelated. Some examples of quasicrossed products are also included. Then, in Section 5, we establish a relation of equivalence between quasicrossed systems, a relation between quasicrossed products and, in the end, we relate them in a suitable way.
Section 6 reveals some compatibility between quasicrossed products and
the Cayley-Dickson process. It is shown that the quasicrossed system corresponding to the deformed group algebra  obtained from the Cayley-Dickson process applied to a deformed group algebra is related to the quasicrossed system corresponding to the initial algebra.
Section 7 is dedicated to simple quasicrossed products. The definition of representation of a graded quasialgebra is introduced and described in a commutative diagram. Some examples of graded modules over graded quasialgebras are included. Finally, in Section 8, central simplicity (as quasialgebras)
for deformed group algebras corresponding to quasicrossed systems is explored.

\section{Preliminaries}

Throughout this work, $A$ denotes an algebra with identity element
$1$ over an algebraically closed field $\mathbb{K}$ with
characteristic  zero. We also fix a multiplicative group $G$ with
neutral element $e$.

\begin{defi}
A \textit{grading} by a group $G$ of an algebra $A$ is a
decomposition  $A =\bigoplus _{g \in G}A_g$ as a direct
sum of vector subspaces $\{A_g \neq 0 : g \in G\}$  of $A$ indexed
by the elements of $G$ satisfying
\begin{equation*}\label{grading_condition_1}
A_gA_h \subset A_{gh}, \hspace*{0.2cm} \hbox{{\rm for any}}
\hspace*{0.2cm} g,h \in G,
\end{equation*}
where we denote by $A_gA_h$ the set of all finite sums of products
$x_gx_h$ with $x_g \in A_g$ and $x_h \in A_h$. An algebra $A$ endowed with a grading by a group $G$ is
 called a \textit{$G$-graded algebra}.
Moreover, if it satisfies the stronger condition
\begin{equation*}
A_gA_h = A_{gh}, \hspace*{0.2cm} \hbox{{\rm for any}}
\hspace*{0.2cm} g,h \in G,
\end{equation*}
we call $A$ a \textit{strongly $G$-graded algebra}.
\end{defi}

\noindent In this paper $G$ is generated by the set of all the
elements $g \in G$ such that $A_g \ne 0$, usually called the
\emph{support} of the grading.
The subspaces $A_g$ are referred to as \textit{homogeneous
component} of the grading. The elements of $\bigcup_{g \in G}A_g$
are called the \textit{homogeneous elements} of $A$. A nonzero
$x_g  \in A_g$ is \textit{homogeneous} of degree $g$ and is denoted
by $deg \: x_g =g$. Any nonzero element $x \in A$ can be written
uniquely in the form $x =\sum_{g \in G}x_g$, where $x_g \in
A_g$
and at most finitely many elements $x_g $ are nonzero.
Given two gradings $\Gamma$ and $\Gamma'$ on $A$, $\Gamma$ is  a
{\it refinement} of $\Gamma'$ if any homogeneous component of
$\Gamma'$ is a (direct) sum of homogeneous components of $\Gamma$.
A grading is {\it fine} if it admits no proper refinement.
Throughout this paper, the gradings will be considered fine.

A subspace $B\subseteq A$ is called a \textit{graded subspace} if $B =\bigoplus_{g \in G}(B\cap A_g)$. Equivalently, a subspace $B$ is graded if
for any $x \in B$, we can write $x =\sum_{g \in G}x_g$, where
$x_g$ is a homogeneous element of degree $g$ in $B$, for any $g
\in G$. The usual regularity concepts will be understood in the
graded sense. That is, a \textit{graded subalgebra} is a
subalgebra which is a graded subspace, and an \textit{ideal} $I
\subset A$ is a graded subspace $I = \oplus_{g \in G}I_g$ of $A$
such that $IA + AI \subset I$.

\begin{defi}\cite{AM1999}
A map $\phi : G \times G \times G \longrightarrow
\mathbb{K}^{\times}$ is called a \textit{cocycle} if
\begin{eqnarray}
&&\phi(h,k,l) \phi(g,hk,l) \phi(g,h,k) = \phi(g,h,kl)
\phi(gh,k,l),
\label{eq:cocyclecond} \\
&&\phi(g,e,h)=1, \label{eq:cocyclecondi}
\end{eqnarray}
hold for any $g,h,k,l \in G$, where $e$ is the identity of $G$.
\end{defi}
\noindent Next lemma lists some properties of
cocycles  useful  in the sequel.
\begin{lem} \label{lem:propcocycle}
If $\phi : G \times G \times G \longrightarrow
\mathbb{K}^{\times}$ is  a cocycle then the following conditions
hold for any $g,h \in G$:
\begin{enumerate}
\item[(i)] $\phi(e,g,h) =\phi(g,h,e) = 1$; \item[(ii)]
$\phi(g,g^{-1},g) \phi(g^{-1},g,h) = \phi(g,g^{-1},gh)$;
\item[(iii)] $\phi(g,g^{-1},g) \phi(g^{-1},g,g^{-1})=1$;
\item[(iv)] $\phi(h,h^{-1},g^{-1}) \phi(g,h,h^{-1})=
\phi(g,h,h^{-1}g^{-1})\phi(gh,h^{-1},g^{-1})$.
\end{enumerate}
\end{lem}

\begin{proof}
First we show (i). In \eqref{eq:cocyclecond} we consider
$k=e$ and get
$$\phi(h,e,l) \phi(g,h,l) \phi(g,h,e) = \phi(g,h,l) \phi(gh,e,l).$$ Now by
 \eqref{eq:cocyclecondi} it comes $\phi(g,h,e) = 1$. We
obtain the other equality in a similar way. To show (ii) we
replace in \eqref{eq:cocyclecond} $h$ by $g^{-1}$, $k$ by $ g$, $l$ by $h$ and take in account (i). The item (iii) is a particular
case of (ii) with $h = g^{-1}$. Now we prove (iv) from the
definition of cocycle. For any $g,h \in G$
\begin{equation*}
\begin{split}
\phi(h,h^{-1},g^{-1}) \phi(g,h,h^{-1})&=
\phi(h,h^{-1},g^{-1})\phi(g,hh^{-1},g^{-1})\phi(g,h,h^{-1})\\
&= \phi(g,h,h^{-1}g^{-1})\phi(gh,h^{-1},g^{-1})
\end{split}
\end{equation*}

\end{proof}
\noindent The category of $G$-graded vector spaces is monoidal by
way of
\begin{equation*}
\begin{split}
\Phi_{V,W,Z}:(V\otimes W) \otimes Z&\longrightarrow  V\otimes (W
\otimes
Z)\\
(v_g \otimes w_h)\otimes z_k &\longmapsto   \phi(g,h,k) v_g
\otimes (w_h \otimes z_k),
\end{split}
\end{equation*}
for any homogeneous elements $v_g$ of degree $g $ in $V$, $w_h$ of
degree $h $ in $W$ and $z_k$ of degree $k $ in $Z$.

\begin{defi}
A map $F : G \times G \longrightarrow \mathbb{K}^{\times}$ is a
\textit{$2$-cochain} if
\begin{equation*}
\begin{split}
&F(e,g) = F(g,e) = 1,
\end{split}
\end{equation*}
holds for any $g \in G$.
\end{defi}

\noindent The notion of graded quasialgebra was introduced in \cite{AM1999}. This new concept includes the usual associative algebras but also some notable nonassociative examples, like the octonions.
\begin{defi}
Let $\phi : G \times G \times G \longrightarrow
\mathbb{K}^{\times}$ be an invertible cocycle. A \emph{$G$-graded
quasialgebra} (or just a \emph{graded quasialgebra}) is a
$G$-graded algebra $A = \bigoplus_{g \in G}A_g$ with pro\-duct map
$ A \otimes A \longrightarrow A$  obeying to the quasiassociative law in the sense
\begin{equation}\label{quassiasociative_condition}
\begin{split}
(x_gx_h)x_k& = \phi(g,h,k)x_g(x_hx_k),
\end{split}
\end{equation}
for any $x_g \in A_g, x_h \in A_h, x_k \in A_k$. Moreover, a
graded quasialgebra $A$ is called \textit{coboundary} if the
associated cocycle is
\begin{equation*}
\begin{split}
\phi(g,h,k)& = \frac{F(g,h)F(gh,k)}{F(h,k)F(g,hk)},
\end{split}
\end{equation*}
for a certain $2$-cochain $F$ and $g,h,k \in G$.
\end{defi}

\begin{rmk} \label{rmk:hfvgyug}
 If $A$ is a
$G$-graded quasialgebra then $A_e$ is an associative algebra ($1 \in A_e$) and $A_g$ is an associative $A_e$-bimodule for any $g \in A_g$.
\end{rmk}

\begin{exa}
All associative graded algebras are graded quasialgebras (with
$\phi(g,h,k) = 1$ for any $g,h,k \in G$). In particular for the
group $G = \mathbb{Z}_2$, the graded quasialgebras admit only two types of algebras.
The mentioned associative case with $\phi $ identically 1, and the antiassociative case with
$\phi(x,y,z) = (-1)^{x y  z}$, for
all $x,y,z \in \mathbb{Z}_2$. The antiassociative quasialgebras were studied in \cite{AEPI2002} and
recently addressed in \cite{ABB2013}. For $G = \mathbb{Z}_3$, every cocycle has the form
\begin{eqnarray*}
&\phi_{111}=\alpha, \hspace{0.2cm} \phi_{112}=\beta, \hspace{0.2cm} \phi_{121}=\frac{1}{\omega\alpha},  \hspace{0.2cm}
\phi_{122}=\frac{\omega}{\beta}, & \\
&\phi_{211}=\frac{\alpha}{\beta\omega}, \hspace{0.2cm} \phi_{212}=\alpha\omega, \hspace{0.2cm} \phi_{221}=\frac{\beta}{\omega\alpha},  \hspace{0.2cm}
\phi_{222}=\frac{\omega}{\alpha}, &
\end{eqnarray*}
for some nonzero $\alpha ,\beta \in \mathbb{K}$ and $\omega$ a cubic root of the unity. Here  $\phi_{111}$ is a shorthand for $\phi(1,1,1)$, etc. $ \mathbb{Z}_n$-quasialgebras are studied in \cite{AM1999a}.
\end{exa}


\begin{lem} \label{lem:strongly_iif}
A $G$-graded quasialgebra $A$ is strongly graded if and only if $1
\in A_gA_{g^{-1}}$ for all $g \in G$.
\end{lem}

\begin{proof}
Suppose $1 \in A_gA_{g^{-1}}$ holds for all $g \in G$. For any $h
\in G$ it follows then that $$A_{gh} = 1A_{gh} \subset
A_gA_{g^{-1}}A_{gh} \subset A_gA_h,$$ hence $A_{gh} = A_gA_h$. The
converse is obvious.
\end{proof}

\begin{lem}
Let $A$ be a strongly graded and commutative quasialgebra, then
$G$ is an abelian group.
\end{lem}

\begin{proof}
Since $A$ is strongly graded, we have that $A_gA_h = A_{gh} \neq
0$ for any $g, h \in G$. Therefore there exist $x_g \in A_g$ and
$x_h \in A_h$ such that $x_gx_h \neq 0$. From $A$ is commutative,
we have that $x_gx_h = x_hx_g \neq 0$, and this implies $gh = hg$.
\end{proof}

\begin{lem}
Let $A$ be a strongly $G$-graded quasialgebra. If $x \in A$ is
such that $xA_g$ = 0 or $A_gx$ = 0, for some $g \in G$, then $x =
0$.
\end{lem}

\begin{proof}
Let $x \in A$ such that $xA_g = 0$ for some $g \in G $  (the
another case is analogue). We have then $xA_gA_{g^{-1}} = 0$, or
equivalently $xA_e = 0$. From $1 \in A_e$, we conclude that $x =
0$.
\end{proof}

\begin{rmk}
By the above result we have for a strongly $G$-graded quasialgebra
$A$ that always is $A_g \neq 0$, for any $g \in G$, even $g$ does
not belong to the support of the grading.
\end{rmk}

\section{Units of a graded quasialgebra}

\begin{defi} An element $u$ of a graded quasialgebra $A$ is called a
\textit{left unit} if there exists a left inverse $u_L^{-1} \in A$
meaning $u_L^{-1} u=1$. Similarly, $u $ is said a \textit{right
unit} if there exists a right inverse $u_R^{-1} \in A$ meaning
$u u_R^{-1} =1$. By an \textit{unit} (or \textit{invertible element})
 we mean an element $u \in A$ such that has a left and
right inverses. We denote by $U(A)$ the set of all units of $A$.
\end{defi}

\begin{defi} We say that a unit $u$ of $A$ is \textit{graded} if $u \in
A_g $ for some $g \in G$. The set of all graded units of $A$
is denoted by $Gr\; U(A)$ and we have  $Gr\; U(A) =
\bigcup_{g \in G}(U(A) \cap A_g)$.
\end{defi}

\begin{rmk} We have that $U(A)$ is not, in
general, a group because the product does not satisfy the
associative law. The left and right inverses of a unit are
not necessary equals.
\end{rmk}

\begin{lem}\label{lem:gradedunit}
Let $u$ be a graded unit of degree $g $ of a graded quasialgebra
$A $. The following assertions hold.
\begin{enumerate}
\item[(i)] The left and right inverses $u_L^{-1}$ and $u_R^{-1}$
of $u$ have degree $g^{-1}$ and are related by  $u_R^{-1} =
\phi(g^{-1},g,g^{-1}) u_L^{-1}$.
\medskip
\item[(ii)] The left and right inverses $u_L^{-1}, u_R^{-1}$ of
$u$ are unique.
\medskip
\item[(iii)] If $w$ is another graded unit of $A$ of degree $h$,
then the product $uw$ is a graded unit of degree $gh$ such that,
\begin{equation*}
\begin{split}
(uw)_L^{-1}= \frac{ \phi(g^{-1},g,h)}{\phi(h^{-1},g^{-1},gh)
}w_L^{-1}u_L^{-1}, \hspace{0.6cm} (uw)_R^{-1}=
\frac{\phi(h,h^{-1},g^{-1})}{\phi(g,h,h^{-1}g^{-1})}w_R^{-1}u_R^{-1}.
\end{split}
\end{equation*}
\item[(iv)] The set $Gr\; U(A)$ is closed under products and
inverses.
\end{enumerate}
\end{lem}
\begin{proof}
(i) We show that $u_L^{-1} \in A_{g^{-1}} $ (it is similar for the
right inverse). We may write $u_L^{-1} =\sum _{h \in G}u_h$, where
$u_h \in A_h$ and at most finitely many elements $u_h $ are
nonzero. From $1=u_L^{-1} u=\sum _{h \in G}u_h u$ it follows that
$u_h =0$ unless $h=g^{-1}$. Thus $u_L^{-1}= u_{g^{-1}} $ has
degree $g^{-1}$. The quasiassociativity of $A$ gives
\begin{equation*}
\begin{split}
u_R^{-1} &=1u_R^{-1} = (u_L^{-1} u)u_R^{-1}=\phi(g^{-1},g,g^{-1})
u_L^{-1} (u u_R^{-1})=\phi(g^{-1},g,g^{-1}) u_L^{-1}
\end{split}
\end{equation*}
\noindent as desired (cf. \cite{AS2002,AS2004}).

\medskip

(ii) Suppose that exist $u_L^{-1}$ and ${u'}_L^{-1}$ two left
inverses of $u$, meaning that $u_L^{-1}u = 1$ and ${u'}_L^{-1}u =
1$. Then $u_L^{-1}u = {u'}_L^{-1}u$. Since $u$ is a unit of $A$,
exists $u_R^{-1}$ satisfying $uu_R^{-1} = 1$. We may write
$(u_L^{-1}u)u_R^{-1} = ({u'}_L^{-1}u)u_R^{-1}$, hence
$\phi(g^{-1},g, g^{-1}) u_L^{-1}(u u_R^{-1}) = \phi(g^{-1},g,
g^{-1}) {u'}_L^{-1}(u u_R^{-1})$ and we obtain $u_L^{-1} =
{u'}_L^{-1}$. The case with the right unit is analogue.

\medskip

(iii) As $A_gA_h \subset A_{gh}$ then $uw$ is a homogeneous
element of degree  $gh $. Since $A$ is quasiassociative, we get
the expression of the left inverse of $uw$ doing
\begin{equation*}
\begin{split}
(w_L^{-1}u_L^{-1})(uw)&= \phi(h^{-1},g^{-1},gh) w_L^{-1}\Bigl(
u_L^{-1}(uw)\Bigr) =
\frac{\phi(h^{-1},g^{-1},gh)}{\phi(g^{-1},g,h)}w_L^{-1}\Bigl((u_L^{-1}
u)w\Bigr)\\
&=\frac{\phi(h^{-1},g^{-1},gh)}{\phi(g^{-1},g,h)}w_L^{-1} w=
\frac{\phi(h^{-1},g^{-1},gh)}{\phi(g^{-1},g,h)}.
\end{split}
\end{equation*}
\noindent  In a similar way we obtain the  right inverse of $uw$
(cf. \cite{AS2002,AS2004}).

\medskip

(iv) By (iii) we conclude that $Gr\; U(A)$ is closed under
products. To show that $Gr\; U(A)$ is closed under inverses,
meaning that whenever $u$ is a graded unit  then $u_L^{-1} $ and
$u_R^{-1} $ are graded units too, we use (i) and observe that
$$\Bigl(\phi(g^{-1},g,g^{-1})u\Bigr) u_L^{-1}=uu_R^{-1}=1$$
and $$u_R^{-1} \Bigl(\frac{1}{\phi(g^{-1},g,g^{-1})}u\Bigr)
=u_L^{-1} u=1$$ completing the proof.
\end{proof}
\begin{rmk}\label{remark_4.5}
From Lemma \ref{lem:gradedunit}(i)-(ii), the left and right
inverses of any $u \in U(A) \cap A_g$ are also units of $A$ and
$$(u_L^{-1})_R^{-1} = u, \hspace{2cm} (u_L^{-1})_L^{-1} =
\phi(g^{-1},g,g^{-1})u,$$
$$(u_R^{-1})_L^{-1} = u, \hspace{2cm} (u_R^{-1})_R^{-1} =
\frac{1}{\phi(g^{-1},g,g^{-1})}u.$$
\end{rmk}
\begin{lem} \label{lem:associative_algebra}
In the case $A$ is a graded associative algebra, left and right
inverses are equal.
\end{lem}
\begin{proof}
It is easy to chek that $u_L^{-1} = u_L^{-1}(uu_R^{-1}) =
(u_L^{-1}u)u_R^{-1}  = u_R^{-1}$ for any $u
\in U(A)$.
\end{proof}

\begin{cor} \label{cor:gradedunitee}
In the case $u \in U(A)\cap A_e$ then its left and right inverses
are equal and belong to $A_e$. Moreover, $U(A) \cap A_e = U(A_e)$.
\end{cor}
\begin{proof}
By Lemma \ref{lem:gradedunit}(i), the left and right inverses of
$u$ belong to $A_e$ and $u_R^{-1} =\phi(e,e,e) u_L^{-1}=u_L^{-1}$.
Therefore $U(A) \cap A_e \subseteq U(A_e)$. The reciprocal
is trivial.
\end{proof}

\begin{rmk}
The map $deg : Gr \; U(A) \rightarrow G$ preserves the multiplication and the elements
$u \in Gr \; U(A) $ such that $ deg \: u = e$ are the set $U(A) \cap A_e =U(A_e)$.
\end{rmk}


\begin{lem}\label{lem:gradedunit2}
\begin{enumerate}

\item[(i)] The map $\mu : Gr \; U(A) \rightarrow Aut(A_e)$
defined by
\begin{equation*}
\begin{split}
\mu(u)(x) = uxu_R^{-1}& \mbox{ for any } u \in Gr \; U(A) \mbox{
and } x \in A_e,
\end{split}
\end{equation*}
satisfies $\mu(uw) = \mu(u) \circ \mu(w)$ for all $u,w \in Gr \;
U(A)$.

\item[(ii)] Right multiplication by any $u \in A_g \cap U(A)$ is
an isomorphism $$A_e \rightarrow A_eu = A_g$$ of left
$A_e$-modules.
\end{enumerate}
\end{lem}

\begin{proof}
(i) Let $u, w$ be two graded units of $A$ such that $deg \: u =
g$ and $ deg \: w = h$. Using Lemma \ref{lem:propcocycle}(iv) we
obtain for any $x \in A_e$,
\begin{equation*}
\begin{split}
\mu(uw)(x) &= (uw)x(uw)_R^{-1} =
(uw)x\frac{\phi(h,h^{-1},g^{-1})}{\phi(g,h,h^{-1}g^{-1})}(w_R^{-1}u_R^{-1})\\
& =
\frac{\phi(h,h^{-1},g^{-1})}{\phi(g,h,h^{-1}g^{-1})}(uw)\Bigl(x(w_R^{-1}u_R^{-1})\Bigr)\\
&
=\frac{\phi(h,h^{-1},g^{-1})}{\phi(g,h,h^{-1}g^{-1})}(uw)\Bigl((xw_R^{-1})u_R^{-1}\Bigr)\\
& =
\frac{\phi(h,h^{-1},g^{-1})}{\phi(g,h,h^{-1}g^{-1})\phi(gh,h^{-1},g^{-1})}\Bigl((uw)(xw_R^{-1})\Bigr)u_R^{-1}\\
& =
\frac{\phi(h,h^{-1},g^{-1})\phi(g,h,h^{-1})}{\phi(g,h,h^{-1}g^{-1})\phi(gh,h^{-1},g^{-1})}
\Bigl(u(wxw_R^{-1})\Bigr)u_R^{-1} \\
& = u(wxw_R^{-1})u_R^{-1} = \mu(u) \circ \mu (w)(x)
\end{split}
\end{equation*}

(ii) First we prove that the right multiplication is a
monomorphism. Let $x,y \in A_e$ such that $xu = yu$. Thus
$(xu)u_R^{-1} =(yu)u_R^{-1}$. Since $x,y \in A_e$ then
$x(uu_R^{-1}) =y(uu_R^{-1})$ and $x = y$. To prove that it
is an epimorphism, we need to see if for any $v \in A_g$ exists $x
\in A_e$ such that $xu = v$. We get it taking $x = vu_R^{-1}$.
\end{proof}

\section{Quasicrossed products and quasicrossed systems}
In this section we introduce the concept of quasicrossed product in the
context of graded quasialgebras.
\begin{defi}
Let $A$ be a $G$-graded quasialgebra. We say that $A$ is a
\textit{quasicrossed product} of $G$ over $A_e$ if for any $g \in G$
exists $\overline{g} \in U(A) \cap A_g$, meaning that, there
exists an unit  $\overline{g}$ in $A$ of any degree $g$.
\end{defi}
\noindent The following examples illustrate that some
important   quasialgebras  are quasicrossed product.
\begin{exa}
Any division graded quasialgebra $A = \oplus_{g \in G}A_g$ is
trivially a quasicrossed pro\-duct of $G$ over $A_e$, because $1 \in
A_e$ and every nonzero homogeneous element is invertible.
\end{exa}

\begin{exa} \label{exa K_FG} Interesting examples of division graded quasialgebras, so of quasicrossed products, are deformed group algebras $\mathbb{K}_FG$ (see \cite{AM1999}). We present properly this class of algebras since we will pay special attention to them in this paper.
Consider the \emph{group algebra} $\mathbb{K}G$, the set of all
linear combinations of elements $\sum_{g \in G}a_gg$, where $a_g
\in \mathbb{K}$ such that $a_g = 0$ for all but
finitely many elements $g$. We define $\mathbb{K}_FG$
with the same underlying vector space as $\mathbb{K}G$ but with a modified
product $g.h := F(g,h)gh$, for any $g,h \in G$, where $F$ is a
2-cochain on $G$. Then $\mathbb{K}_FG$ is a coboundary graded
quasialgebra. Moreover, any $\mathbb{K}_FG$ is a quasicrossed product.
In fact, given $g \in G$ and $a_g \in \mathbb{K}^{\times}$ then
the homogeneous element $a_gg \in (\mathbb{K}_FG)_g$ is an unit with
left inverse and right inverse:
\begin{equation*}
\begin{split}
(a_gg)_L^{-1} = \phi(g,g^{-1},g)(a_gg)_R^{-1} =
F(g^{-1},g)^{-1} a_g^{-1} g^{-1}.
\end{split}
\end{equation*}
There are two classes of modified group algebras particularly
interesting, namely the Cayley algebras and the Clifford algebras. We
mention just some well studied Cayley algebras:

\begin{enumerate}
\item The complex algebra $\mathbb{C}$ is a $\mathbb{K}_FG$
quasialgebra with $G = \mathbb{Z}_2$ and $F(x,y) = (-1)^{xy}$, for
$x,y \in \mathbb{Z}_2$.

\item The quaternion algebra $\mathbb{H}$ is a $\mathbb{K}_FG$
quasialgebra with $G = \mathbb{Z}_2 \times \mathbb{Z}_2$ and $
F(\overrightarrow{x},\overrightarrow{y}) =
(-1)^{x_1y_1+(x_1+x_2)y_2}$, where $\overrightarrow{x} = (x_1,x_2)
\in \mathbb{Z}_2 \times \mathbb{Z}_2$ is a vector notation.

\item The octonion algebra $\mathbb{O}$ is another $\mathbb{K}_FG$
algebra for $G = \mathbb{Z}_2 \times \mathbb{Z}_2 \times
\mathbb{Z}_2$ and $F(\overrightarrow{x},\overrightarrow{y}) =
(-1)^{\sum_{i \leq j}x_iy_j+y_1x_2x_3+x_1y_2x_3+x_1x_2y_3}$, where
$\overrightarrow{x} = (x_1,x_2,x_3) \in \mathbb{Z}_2 \times
\mathbb{Z}_2 \times \mathbb{Z}_2$.
\end{enumerate}
Any Clifford algebra is another quasialgebra of type
$\mathbb{K}_FG$ for $G = (\mathbb{Z}_2)^n$ and $2$-cochain
$F(\overrightarrow{x},\overrightarrow{y}) = (-1)^{\sum_{i \leq
j}x_iy_j}$ where $\overrightarrow{x} = (x_1,...,x_n) \in
(\mathbb{Z}_2)^n$. Recall that $\mathbb{C}$ and $\mathbb{H}$ are
both Cayley and Clifford algebras.

\end{exa}

\begin{exa}
 Let $\displaystyle
\mbox{Mat}_n(\Delta)$ be the $\displaystyle \mathbb{Z}_2$-graded
algebra  of the $n\times n$ matrices over $\Delta$ with the
natu\-ral $\displaystyle \mathbb{Z}_2$-gradation inherited from
$\displaystyle \Delta$, where $\displaystyle
{\Delta}={\Delta}_{\bar 0} \oplus {\Delta}_{\bar 1}$ is   a
division antiassociative quasialgebra  ( $\simeq  \langle D,
\sigma, a \rangle$ see \cite{AEPI2002})  and $n \in \mathbb{N}$.
 Consider $\displaystyle
\mbox{Mat}_n(\Delta)= \mbox{Mat}_n(\Delta_{\bar 0})\oplus
\mbox{Mat}_n(\Delta_{\bar 0})u$  equipped with multiplication
defined by
\begin{eqnarray*}
\displaystyle
&&A(Bu)=(AB)u\\
&&(Au)B=(A\overline{B})u ,\\
&&(Au)(Bu)=aA\overline{B}, \hspace*{1cm} \mbox{for all}
\hspace*{0.2cm} {A,B \in \mbox{Mat}_n(\Delta_{\bar 0})},
\end{eqnarray*}
where the matrix $\displaystyle \overline{B}$ is obtained from the
matrix $\displaystyle B=[b_{ij}]_{1 \leq i,j\leq n}$ by replacing
the term $\displaystyle b_{ij}$ by $\displaystyle \sigma (
b_{ij})$, for all $\displaystyle i,j \in \{1 ,\ldots, n\}$. Then
the simple antiassociative quasialgebra $\displaystyle
\mbox{Mat}_n(\Delta)$ is clearly a quasicrossed product of
$\mathbb{Z}_2$ over $\displaystyle \mbox{Mat}_n(\Delta_{\bar 0})$. It is clear that $id$ is an unit in $\mbox{Mat}_n(\Delta_{\bar 0})$ and $id u$ is an unit in $\mbox{Mat}_n(\Delta_{\bar 0})u$.

For $n,m \in \mathbb{N}$, the set
$\widetilde{Mat}_{n,m}(D)$ of $(n+m) \times (n+m)$ matrices over a
division algebra $D$, with the chess board $\mathbb{Z}_2$-grading:
\begin{eqnarray*}
&&  \displaystyle {\widetilde{\mbox{Mat}}_{n,m}(D)}_{\bar
0}:=\left\{ \left(
                                                    \begin{array}{cc}
                                                      a & 0 \\
                                                      0 & b \\
                                                    \end{array}
                                                  \right) :
                                                  a \in
\mbox{Mat}_{n}(D), b \in  \mbox{Mat}_{m}(D)
\right\}\\
&&{\widetilde{\mbox{Mat}}_{n,m}(D)}_{\bar 1}:=\left\{ \left(
                                                    \begin{array}{cc}
                                                     0 & v \\
                                                      w & 0 \\
                                                    \end{array}
                                                  \right):
                                                  v \in
\mbox{Mat}_{n\times m}(D), w \in \mbox{Mat}_{m\times n}(D)
\right\},
\end{eqnarray*}
and with multiplication given by
\begin{equation*}
\begin{split}
\left(%
\begin{array}{cc}
 a_1 & v_1 \\
 w_1 & b_1 \\
\end{array}%
\right) \cdot \left(%
\begin{array}{cc}
 a_2 & v_2 \\
 w_2 & b_2 \\
\end{array}%
\right) = \left(%
\begin{array}{cc}
 a_1a_2+v_1w_2 & a_1v_2+v_1b_2 \\
 w_1a_2+b_1w_2 & -w_1v_2+b_1b_2 \\
\end{array}%
\right)
\end{split}
\end{equation*}
is a quasicrossed product. Indeed, let $a \in Mat_n(D)$ and $ b \in
Mat_m(D)$ be two invertible matrices,
then $\left(%
\begin{array}{cc}
 a & 0 \\
 0 & b \\
\end{array}%
\right)$ is an unit in $\widetilde{Mat}_{n,m}(D)_{\overline{0}}$
with
\begin{equation*}
\begin{split}
\left(%
\begin{array}{cc}
 a & 0 \\
 0 & b \\
\end{array}%
\right)_R^{-1} = \left(%
\begin{array}{cc}
 a & 0 \\
 0 & b \\
\end{array}%
\right)_L^{-1} = \left(%
\begin{array}{cc}
 a^{-1} & 0 \\
 0 & b^{-1} \\
\end{array}%
\right).
\end{split}
\end{equation*}
Also, if $v \in Mat_{n \times m}(D)$ and $ w \in Mat_{m \times
n}(D)$
are invertible matrices, then $\left(%
\begin{array}{cc}
 0 & v \\
 w & 0 \\
\end{array}%
\right)$ is an unit in $\widetilde{Mat}_{n,m}(D)_{\overline{1}}$
with
\begin{equation*}
\begin{split}
\left(%
\begin{array}{cc}
 0 & v \\
 w & 0 \\
\end{array}%
\right)_R^{-1} = -\left(%
\begin{array}{cc}
 0 & v \\
 w & 0 \\
\end{array}%
\right)_L^{-1} = \left(%
\begin{array}{cc}
 0 & -w^{-1} \\
 v^{-1} & 0 \\
\end{array}%
\right)
\end{split}
\end{equation*}
Actually, any unital antiassociative quasialgebra with semisimple even part, except the trivial one, is a quasicrossed product, because it is a finite direct sum of ideals of the form $\displaystyle
\mbox{Mat}_{n_i}(\Delta^i)$ for some number $n_i$ and some division antiassociative quasialgebra  $\Delta^i$, and ideals of the form  $\widetilde{Mat}_{n_j,m_j}(D^j)$ for some
division algebra $D^j$ and some numbers $n_j$ and $m_j$ (see \cite[Theorem 10]{AEPI2002}). As naturally expected, these examples show that there are quasicrossed products which are not division graded quasialgebras.
\end{exa}

\begin{exa} \label{exa:deformed.matrices}
Consider the $\mathbb{Z}_n$-graded quasialgebra  of the deformed matrices $M_{n,\phi}(\mathbb{K})$  of the
usual $n \times n$ matrices $M_n(\mathbb{K})$ with the basis elements
$E_{ij}$ of degree $j - i$, for $i, j \in \mathbb{Z}_n$, and the multiplication
\begin{equation*}
\begin{split}
(X \cdot Y)_{ij} = \sum\limits_{k=1}^n
\frac{\phi(i,-k,k -j)}{\phi(-k,k,-j)}X_{ik}Y_{kj},
\end{split}
\end{equation*}
for any $X=(X_{ij})$ and $Y=(Y_{ij})$ in  $   M_n (\mathbb{K})$ (cf. in \cite{AS2006}). This $\mathbb{Z}_n$-graded quasialgebra  is a quasicrossed product. Indeed, we easily find an invertible element in  each homogeneous component of $ M_{n,\phi}(\mathbb{K}) $.
\end{exa}

\begin{rmk}
Observe that not all graded quasialgebras are quasicrossed products. For example, we can easily extract subalgebras of the algebra of  Example \ref{exa:deformed.matrices} which are not quasicrossed products. The subalgebra $T_{n,\phi}(\mathbb{K})$
of the $\mathbb{Z}_n$-graded quasialgebra $M_{n,\phi}(\mathbb{K})$ formed by the upper triangular matrices  is not a quasicrossed product. For example, the 1-dimensional homogeneous component  $(T_{n,\phi}(\mathbb{K}))_{n-1}$ with basis $\{ E_{1n}\} $ does not contain an invertible element.
\end{rmk}
\begin{defi}\label{def_quasicrossed_mapping}
Assume that $B$ is an associative algebra. Given maps
$$\sigma : G \rightarrow Aut(B)\hspace{1cm} \mbox{ \textit{automorphism
system}}$$
$$\alpha : G \times G \rightarrow U(B) \hspace{1cm} \mbox{ \textit{quasicrossed
mapping}}$$ and a cocycle $\phi : G \times G \times G
\longrightarrow \mathbb{K}^{\times}$, we say that
$(G,B,\phi,\sigma,\alpha)$ is a \textit{quasicrossed system} for $G$
over $B$ if the following properties hold:
\begin{eqnarray}
&&\sigma(g)\Bigl(\sigma(h)(x)\Bigr) =
\alpha(g,h)\sigma(gh)(x)\alpha(g,h)^{-1} \label{cond_quasicrossed_system1} \\
&&\alpha(g,h)\alpha(gh,k) =
\phi(g,h,k)\sigma(g)(\alpha(h,k))\alpha(g,hk)
\label{cond_quasicrossed_system2} \\
&&\alpha(g,e) = \alpha(e,g) = 1 \label{cond_quasicrossed_system3}
\end{eqnarray}
for any $g,h,k \in G$ and $x \in B$.
\end{defi}

\noindent Let $A$ be a graded quasialgebra which is a quasicrossed
product of $G$ over $A_e$. Then for any $g \in G$ exists an unit
$\overline{g} \in U(A) \cap A_g$ with $\overline{e} = 1$, and
define a map $\sigma(g) : A_e \longrightarrow A_e$ by
\begin{equation}\label{eq:sigma_condition}
\sigma(g)(x) := \overline{g}x{\overline{g}_R}^{-1} \hspace{0.6cm}
\mbox{for any }x \in A_e.
\end{equation}

\begin{lem}\label{sigma_product}
For any $g \in G$, $\sigma(g)$ is an automorphism of $A_e$, meaning that for $x,y \in A_e$
\begin{equation*}
\begin{split}
&\sigma(g)(xy) = \sigma(g)(x)\sigma(g)(y).
\end{split}
\end{equation*}
\end{lem}
\begin{proof}
For any $g \in G$, as $\overline{g}$ is a unit it is obvious that the map $\sigma(g)$ is bijective.
Applying Lemma \ref{lem:gradedunit}(i), we obtain for any $g \in G$ and $x,
y \in A_e$
\begin{equation*}
\begin{split}
\sigma(g)(xy)&= \overline{g}xy\overline{g}_R^{-1} =
\Bigl(\overline{g}x(\overline{g}_L^{-1}\overline{g})\Bigr)y\overline{g}_R^{-1}
=
\frac{1}{\phi(g,g^{-1},g)}\Bigl((\overline{g}x\overline{g}_L^{-1})\overline{g}\Bigr)y\overline{g}_R^{-1}\\
& =
\frac{1}{\phi(g,g^{-1},g)}(\overline{g}x\overline{g}_L^{-1})(\overline{g}y\overline{g}_R^{-1})
=(\overline{g}x\overline{g}_R^{-1})(\overline{g}y\overline{g}_R^{-1})
=\sigma(g)(x)\sigma(g)(y)
\end{split}
\end{equation*}
as desired.
\end{proof}

\begin{pro}\label{Prop_page_176}
Let $A$ be a graded quasialgebra  which is a quasicrossed product of
$G$ over $A_e$. For any $g \in G$, fix a unit $\overline{g} $ in $
A_g$ with $\overline{e} = 1$. Let $\sigma: G \rightarrow Aut(A_e)$
be the corresponding automorphism system given by Equation
\eqref{eq:sigma_condition} and  $\alpha : G \times G
\rightarrow U(A_e)$  be defined by

\begin{equation}\label{alpha_condition}
\alpha(g,h) =(\overline{g}\hspace{0.1cm}\overline{h})
(\overline{gh})_R^{-1}=
\phi((gh)^{-1},gh,(gh)^{-1})(\overline{g}\hspace{0.1cm}\overline{h})
(\overline{gh})_L^{-1},
\end{equation}
for any $g,h \in G$. Then the following properties hold:
\begin{enumerate}
\item[(i)] $A$ is a strongly graded quasialgebra with $A_g =
A_e\overline{g} = \overline{g}A_e$.

\item[(ii)] $(G,A_e,\phi,\sigma,\alpha)$ is a quasicrossed system for
$G$ over $A_e$ (to which we refer as corres\-ponding to $A$).

\item[(iii)] $A$ is a free (left or right) $A_e$-module freely
generated by the elements $\overline{g}$, where $g \in G$.

\item[(iv)] For all $g, h \in G$ and $x,y \in A_e$,
\begin{equation}\label{item4}
(x\overline{g})(y\overline{h}) =
x\sigma(g)(y)\alpha(g,h)\overline{gh}.
\end{equation}
\end{enumerate}
Conversely, for any associative algebra $B$ and any quasicrossed system
$(G,B,\phi, \sigma,\alpha)$ for $G$ over $B$, the free $B$-module
$C$ freely generated by the elements $\overline{g}$, for $g \in
G$, with multiplication given by Equation \eqref{item4} (with $x,y
\in B)$ is a $G$-graded quasialgebra (with $C_g = B\overline{g}$
for all $g \in G$) which is a quasicrossed product of $G$ over $C_e =
B$ and having $(G,B,\phi,\sigma,\alpha)$ as a corresponding
quasicrossed system.
\end{pro}

\begin{rmk}
We note that Proposition \ref{Prop_page_176} generalizes the results on quasiassociative division algebras presented by H. Albuquerque and A. Santana (see  Theorem 1.1 in \cite{AS2004} and Theorem 3.2 in \cite{AS2006}). The quasiassociative division algebras are precisely the quasicrossed products over the division associative algebras. Moreover, the three identities defining the multiplication in quasiassociative division algebras are now condensed in equation \eqref{item4}.
\end{rmk}

\begin{proof}
(i) Let $g \in G$ and take $u \in U(A) \cap A_g$. By
Lemma \ref{lem:gradedunit}(i), $$u_L^{-1}, u_R^{-1} \in
A_{g^{-1}}$$ and therefore $1 = u_L^{-1}u \in A_{g^{-1}}A_g $ and
$1 = uu_R^{-1} \in A_gA_{g^{-1}}$. By Lemma
\ref{lem:strongly_iif}, we conclude that $A$ is a strongly graded
quasialgebra. Applying Lemma \ref{lem:gradedunit2}(iii), $A_g =
A_e\overline{g}$ and the argument of this lemma applied to left
multiplication shows that $A_g = \overline{g}A_e$, proving this
item.

(ii) First we prove condition (\ref{cond_quasicrossed_system1}). Let $g, h \in G$ and $x \in A_e$. We have
\begin{equation*}
\begin{split}
\sigma(g)(\sigma(h)(x)) &=
\sigma(g)\Bigl(\overline{h}x\overline{h}_R^{-1}\Bigr) =
\Bigl(\overline{g}(\overline{h}(x\overline{h}_R^{-1}))\Bigr)\overline{g}_R^{-1}
=
\frac{1}{\phi(g,h,h^{-1})}\Bigl((\overline{g}\hspace{0.1cm}\overline{h})(x\overline{h}_R^{-1})\Bigr)\overline{g}_R^{-1}\\
& =
\frac{\phi(gh,h^{-1},g^{-1})}{\phi(g,h,h^{-1})}(\overline{g}\hspace{0.1cm}\overline{h})\Bigl((x\overline{h}_R^{-1})
\overline{g}_R^{-1}\Bigr) =
\frac{\phi(gh,h^{-1},g^{-1})}{\phi(g,h,h^{-1})}(\overline{g}\hspace{0.1cm}\overline{h})x(\overline{h}_R^{-1}\overline{g}_R^{-1})
\end{split}
\end{equation*}
\noindent and we get
\begin{equation}\label{1}
\begin{split}
\sigma(g)(\sigma(h)(x)) =
\frac{\phi(gh,h^{-1},g^{-1})}{\phi(g,h,h^{-1})}(\overline{g}\hspace{0.1cm}\overline{h})x(\overline{h}_R^{-1}\overline{g}_R^{-1}).
\end{split}
\end{equation}
\noindent Now, using Lemma \ref{lem:gradedunit}  we observe
that
\begin{equation*}
\begin{split}
(\overline{gh} )_R^{-1}(\alpha(g,h))_R^{-1}
&=
(\overline{gh} )_R^{-1}\Bigl(\phi((gh)^{-1},gh,(gh)^{-1})(\overline{g}\hspace{0.1cm}\overline{h})
(\overline{gh})_L^{-1}\Bigl)_R^{-1}\\
&\hspace*{-1cm}=
\frac{\phi((gh)^{-1},gh,(gh)^{-1})}{\phi((gh)^{-1},gh,(gh)^{-1})\phi(gh,(gh)^{-1},gh(gh)^{-1})}
(\overline{gh})_R^{-1}\Bigl(((\overline{gh} )_L^{-1})_R^{-1}(\overline{g}\hspace{0.1cm}\overline{h})_R^{-1}\Bigr)
\\
&\hspace*{-1cm}=
(\overline{gh})_R^{-1}\Bigl(\overline{gh}(\overline{g}\hspace{0.1cm}\overline{h})_R^{-1}\Bigr)=
\frac{\phi((gh)^{-1},gh,(gh)^{-1}) }{\phi((gh)^{-1},gh,(gh)^{-1})
}
\Bigl((\overline{gh})_L^{-1}\overline{gh}\Bigr)(\overline{g}\hspace{0.1cm}\overline{h})_R^{-1}\\
&\hspace*{-1cm}=\Bigl((\overline{gh})_L^{-1}\overline{gh}\Bigr)(\overline{g}\hspace{0.1cm}\overline{h})_R^{-1}=
\frac{\phi(h,h^{-1},g^{-1})}{\phi(g,h,h^{-1}g^{-1})}\overline{h}_R^{-1}\overline{g}_R^{-1}
\end{split}
\end{equation*}
\noindent and hence
\begin{equation}\label{one}
\begin{split}
\overline{h}_R^{-1}\overline{g}_R^{-1} =
\frac{\phi(g,h,h^{-1}g^{-1})}{\phi(h,h^{-1},g^{-1})}(\overline{gh} )_R^{-1}(\alpha(g,h))_R^{-1}.
\end{split}
\end{equation}
\noindent Applying Lemma \ref{lem:propcocycle}(iii) we obtain
\begin{equation*}
\begin{split}
\alpha(g,h)\overline{gh} &=
\Bigl(\phi((gh)^{-1},gh,(gh)^{-1})(\overline{g}\hspace{0.1cm}\overline{h})(\overline{gh})_L^{-1}\Bigl)\overline{gh}\\
&
=\phi((gh)^{-1},gh,(gh)^{-1})\phi(gh,(gh)^{-1},gh)(\overline{g}\hspace{0.1cm}\overline{h})
\Bigl((\overline{gh})_L^{-1}\overline{gh}\Bigl) =
\overline{g}\hspace{0.1cm}\overline{h}
\end{split}
\end{equation*}
\noindent then
\begin{equation}\label{two}
\begin{split}
\alpha(g,h)\overline{gh} = \overline{g}
\hspace{0.1cm}\overline{h}.
\end{split}
\end{equation}
\noindent Returning to \eqref{1}, using \eqref{one} and
\eqref{two} we have
\begin{equation*}
\begin{split}
\sigma(g)(\sigma(h)(x)) =
\frac{\phi(gh,h^{-1},g^{-1})\phi(g,h,h^{-1}g^{-1})}{\phi(g,h,h^{-1})\phi(h,h^{-1},g^{-1})}
\alpha(g,h)\overline{gh}x(\overline{gh})_R^{-1}(\alpha(g,h))_R^{-1}
\end{split}
\end{equation*}
\noindent and by Lemma \ref{lem:propcocycle}(iv) we conclude
\begin{equation*}
\begin{split}
\sigma(g)(\sigma(h)(x)) =
\alpha(g,h)\sigma(gh)(x)(\alpha(g,h))_R^{-1}
\end{split}
\end{equation*}
proving (\ref{cond_quasicrossed_system1}). Let us now care about
(\ref{cond_quasicrossed_system2}). For any $g,h,k \in G$, by Lemma
\ref{sigma_product} and using condition \eqref{two} we have
\begin{equation*}
\begin{split}
(\overline{g} \hspace{0.1cm} \overline{h})\overline{k} & =
(\alpha(g,h)\overline{gh})\overline{k} = \alpha(g,h)(\overline{gh}
\hspace{0.1cm}\overline{k}) =
\alpha(g,h)\alpha(gh,k)\overline{gh}\overline{k}
\end{split}
\end{equation*}
\noindent On the other hand, we obtain
\begin{equation}\label{eq:ttt}
\begin{split}
\overline{g}(\overline{h} \hspace{0.1cm} \overline{k}) =
\overline{g}(\alpha(h,k)\overline{hk})=(\overline{g}\alpha(h,k))\overline{hk}
\end{split}
\end{equation}
\noindent using Lemma \ref{lem:propcocycle}(iii) we have
\begin{equation*}
\begin{split}
\sigma(g)(\alpha(h,k))\overline{g}&= \Bigr(\overline{g}
\alpha(h,k){\overline{g}_R}^{-1}\Bigr)\overline{g} =
\phi(g^{-1},g,g^{-1}) \Bigr(\overline{g} \alpha(h,k)
\overline{g_L}^{-1}\Bigr)\overline{g}\\
&= \phi(g^{-1},g,g^{-1})\phi(g,g^{-1},g) \Bigl(\overline{g}
\alpha(h,k)\Bigr)\Bigl({\overline{g}_L}^{-1}\overline{g}\Bigr) =
\overline{g}\alpha(h,k)
\end{split}
\end{equation*}
Returning to \eqref{eq:ttt}
\begin{equation*}
\begin{split}
\overline{g}(\overline{h} \hspace{0.1cm} \overline{k}) &=
(\overline{g}\alpha(h,k))\overline{hk} =
\Bigl(\sigma(g)(\alpha(h,k))\overline{g}\Bigr)\overline{hk}=
\sigma(g)(\alpha(h,k))(\overline{g} \hspace{0.1cm}\overline{hk}) \\
&= \sigma(g)(\alpha(h,k))\alpha(g,hk)\overline{ghk}.
\end{split}
\end{equation*}
\noindent Since $G$ is associative and
$(\overline{g}\hspace{0.1cm}\overline{h})\overline{k} =
\phi(g,h,k) \overline{g}(\overline{h}\hspace{0.1cm}\overline{k})$,
we conclude that
\begin{equation*}
\begin{split}
\alpha(g,h)\alpha(gh,k) = \phi(g,h,k) \sigma(g)
(\alpha(h,k))\alpha(g,hk)
\end{split}
\end{equation*}
proving (\ref{cond_quasicrossed_system2}). Because $\overline{e} = 1$
we have
\begin{equation*}
\begin{split}
\alpha(g,e)& =
\phi((ge)^{-1},ge,(ge)^{-1})(\overline{g}\hspace{0.1cm}\overline{e})(\overline{ge})_L^{-1}
= \phi(g^{-1},g,g^{-1})\overline{g}\hspace{0.1cm}\overline{g}_L^{-1} \\
&= \frac{\phi(g^{-1},g,g^{-1})}{\phi(g^{-1},g,g^{-1})}
\overline{g}\hspace{0.1cm}\overline{g}_R^{-1} = 1
\end{split}
\end{equation*}
thus (\ref{cond_quasicrossed_system3}) is also true, proving (ii).

(iii) It is a direct consequence of (i).

(iv) Let $g, h \in G$ and $x,y \in A_e$. Using Lemma
\ref{lem:propcocycle}(ii) we obtain
\begin{equation*}
\begin{split}
(x\overline{g})(y\overline{h}) &=
(x\overline{g})\Bigl(\Bigl(y(\overline{g}_L^{-1}\overline{g})\Bigr)\overline{h}\Bigr)
=
(x\overline{g})\Bigl(\Bigl((y\overline{g}_L^{-1})\overline{g}\Bigr)\overline{h}\Bigr)
\\
&= \phi(g^{-1},g,h) (x\overline{g})
\Bigl((y\overline{g}_L^{-1})(\overline{g}\hspace{0.1cm}\overline{h})\Bigr)
= \frac{\phi(g^{-1},g,h)}{\phi(g,g^{-1},gh)}
\Bigl((x\overline{g})(y\overline{g}_L^{-1})\Bigr)(\overline{g}\hspace{0.1cm}\overline{h})
\\
&= \frac{\phi(g^{-1},g,h)}{\phi(g,g^{-1},gh)\phi(g^{-1},g,g^{-1})}
\Bigl((x\overline{g})(y\overline{g}_R^{-1})\Bigr)(\overline{g}\hspace{0.1cm}\overline{h})\\
&= x\sigma(g)(y)(\overline{g}\hspace{0.1cm}\overline{h}) =
x\sigma(g)(y)\alpha(g,h)\overline{gh}
\end{split}
\end{equation*}
\noindent proving \eqref{item4}. To prove the converse, we need
only to verify that the multiplication given by \eqref{item4} is
quasiassociative. To this end, let $g,h,k \in G$ and $x,y,z \in
B$. As $\sigma(g) \in Aut(B)$ we have
\begin{equation*}
\begin{split}
(x\overline{g})\Bigl((y\overline{h})(z\overline{k})\Bigr) &=
(x\overline{g})\Bigl(y\sigma(h)(z)\alpha(h,k)\overline{hk}\Bigr)\\
&=
x\sigma(g)\Bigl(y\sigma(h)(z)\alpha(h,k)\Bigl)\alpha(g,hk)\overline{ghk}
\\
&=
x\sigma(g)(y)\sigma(g)(\sigma(h)(z))\sigma(g)(\alpha(h,k))\alpha(g,hk)\overline{ghk}
\\
&\stackrel{\eqref{cond_quasicrossed_system1}}{=}
x\sigma(g)(y)\alpha(g,h)\sigma(gh)(z)\alpha(g,h)^{-1}\sigma(g)(\alpha(h,k))\alpha(g,hk)\overline{ghk}\\
&\stackrel{\eqref{cond_quasicrossed_system2}}{=} \frac{ 1}{
\phi(g,h,k)}
x\sigma(g)(y)\alpha(g,h)\sigma(gh)(z)\alpha(g,h)^{-1}\alpha(g,h)\alpha(gh,k)\overline{ghk}
\\
&= \frac{ 1}{ \phi(g,h,k)}
x\sigma(g)(y)\alpha(g,h)\sigma(gh)(z)\alpha(gh,k)\overline{ghk}
.
\end{split}
\end{equation*}
\noindent On the other hand,
\begin{equation*}
\begin{split}
\Bigl((x\overline{g})(y\overline{h})\Bigr)(z\overline{k}) =
\Bigl(x\sigma(g)(y)\alpha(g,h)\overline{gh}\Bigr)(z\overline{k}) =
x\sigma(g)(y)\alpha(g,h)\sigma(gh)(z)\alpha(gh,k)\overline{ghk}
\end{split}
\end{equation*}
\noindent therefore,
\begin{equation*}
\begin{split}
\Bigl((x\overline{g})(y\overline{h})\Bigr)(z\overline{k})=\phi(g,h,k)
(x\overline{g})\Bigl((y\overline{h})(z\overline{k})\Bigr)
\end{split}
\end{equation*}
completing the proof.
\end{proof}

\section{Equivalence of quasicrossed products and quasicrossed systems}
In this section we establish a relation of equivalence between quasicrossed systems, a relation between quasicrossed products and, in the end, we relate the two notions.
\begin{defi}
We say that two quasicrossed systems $(G,B, \phi,\sigma,\alpha)$ and
$(G,B,\phi,\sigma',\alpha')$ over an associative algebra $B$ for a fixed  cocycle $\phi : G \times G \times G \longrightarrow
\mathbb{K}^{\times}$ are
\textit{equivalent} if there exists a map $u : G \rightarrow U(B)$
with $u(e) = 1$ such that
\begin{eqnarray}\label{relation_sigma_sigma'}
&&\sigma'(g) = i_{u(g)} \circ \sigma(g) \\
\label{relation_alpha_alpha'} &&\alpha'(g,h) =
u(g)\sigma(g)(u(h))\alpha(g,h)u(gh)^{-1},
\end{eqnarray}
\noindent for any  $g,h \in G$, where $i_{y}(x)=yxy^{-1}$ for $x \in B$ and $y \in U(B)$.
\end{defi}

We define an equivalence relation in the class of quasicrossed systems over an associative algebra $B$ for a fixed  cocycle $\phi : G \times G \times G \longrightarrow
\mathbb{K}^{\times}$.
Assume that a graded quasialgebra $A$ is a quasicrossed product of $G$
over $A_e$. Due to Proposition \ref{Prop_page_176}, any choice of
a unit $\overline{g}$ of $A$ in $A_g,$ for any $ g \in G$, with
$\overline{e} = 1$, determines a corresponding quasicrossed system
$(G,A_e,\phi,\sigma,\alpha)$ for $G$ over $A_e$, with $\sigma$ and
$\alpha$ given by
\begin{eqnarray*}
&&\sigma(g)(x) = \overline{g}x\overline{g}_R^{-1}
\label{determines_quasicrossed_system} \\
&&\alpha(g,h) = (\overline{g} \hspace{0.1cm}
\overline{h})(\overline{gh})_R^{-1},
\label{determines_quasicrossed_system2}
\end{eqnarray*}
\noindent for any $x \in A_e$ and $g,h \in G$. Now, let
$\{\widetilde{g} : g \in G\}$ be another such set of units and
$(G,A_e,\phi,\sigma',\alpha')$ be the corres\-ponding quasicrossed
system. Because $\widetilde{g} \in A_g$, we infer from Proposition
\ref{Prop_page_176}(i) that there is a function $u : G \rightarrow
U(A_e)$ with $u(e) = 1$ such that $$\widetilde{g} =
u(g)\overline{g} \hspace{0.6cm} \mbox{for all} \hspace{0.2cm} g
\in G.$$ We note that $u(g)$ is indeed an unit of $ A_e $ with
inverse $ u(g)^{-1}=\overline{g}\widetilde{g}_R^{-1}$.
\begin{lem} In the previous conditions
we have that the quasicrossed systems $(G,A_e,\phi,\sigma,\alpha)$ and
$(G,A_e,\phi,\sigma',\alpha')$ are equivalent over the associative algebra $A_e$.
\end{lem}
\begin{proof}
For $g \in G$ and $x \in A_e$ we have
\begin{equation*}
\begin{split}
\sigma'(g)(x) &=\widetilde{g}x\widetilde{g}_R^{-1}=
u(g)\overline{g} x (u(g)\overline{g})_R^{-1} = u(g)\overline{g} x
\overline{g}_R^{-1}u(g)^{-1} =
u(g)(\overline{g}x\overline{g}_R^{-1})u(g)^{-1} \\
&= u(g)\sigma(g)(x)u(g)^{-1} = i_{u(g)}(\sigma(g)(x))
\end{split}
\end{equation*}
\noindent proving \eqref{relation_sigma_sigma'}.
Now we take care of \eqref{relation_alpha_alpha'}. For $g, h \in
G$, using Lemma \ref{lem:propcocycle}(ii)-(iii) we have
\begin{equation*}
\begin{split}
u(gh)\overline{gh} &=  \widetilde{gh}  =
\alpha'(g,h)^{-1}\widetilde{g}\hspace{0.1cm} \widetilde{h}   =
\alpha'(g,h)^{-1}u(g)\overline{g}u(h)\overline{h} =
\alpha'(g,h)^{-1}(u(g)\overline{g})(u(h)(\overline{g}_L^{-1}\overline{g})\overline{h})\\
& = \frac{1}{\phi(g^{-1},g,g^{-1})}
\alpha'(g,h)^{-1}(u(g)\overline{g})(u(h)(\overline{g}_R^{-1}\overline{g})\overline{h})\\
& = \frac{\phi(g^{-1},g,h)}{\phi(g^{-1},g,g^{-1}) }
\alpha'(g,h)^{-1}(u(g) \overline{g})(u(h)\overline{g}_R^{-1}
(\overline{g}\hspace{0.1cm}\overline{h}))\\
& =
\frac{\phi(g^{-1},g,h)}{\phi(g^{-1},g,g^{-1})\phi(g,g^{-1},gh)}
\alpha'(g,h)^{-1}u(g)\Bigl(\overline{g}u(h)\overline{g}_R^{-1}\Bigr)(\overline{g}\hspace{0.1cm}\overline{h})\\
& =
\alpha'(g,h)^{-1}u(g)\sigma(g)(u(h))(\overline{g}\hspace{0.1cm}\overline{h})\\
& = \alpha'(g,h)^{-1}u(g)\sigma(g)(u(h))\alpha(g,h)\overline{gh}
\end{split}
\end{equation*}
\noindent therefore
\begin{equation*}
\begin{split}
\alpha'(g,h) = u(g)\sigma(g)(u(h))\alpha(g,h)u(gh)^{-1}
\end{split}
\end{equation*}
\noindent proving \eqref{relation_alpha_alpha'}.  Consequently,
$(G,A_e,\phi,\sigma,\alpha)$ and $(G,A_e,\phi,\sigma',\alpha')$
are equivalent as desired.
\end{proof}
\noindent Thus any given graded quasialgebra $A$ which is a
quasicrossed product of $G$ over $A_e$ determines a unique
equiva\-lence class of corresponding quasicrossed systems for $G$ over
$A_e$. We stress the independence of the choice of the sets of units used to define the quasicrossed systems.

\begin{defi}
Assume that $A, A'$ are two quasicrossed products of $G$ over $A_e$. We
say that $A$ and $A'$ are \textit{equivalent} if there is a graded
isomorphism of algebras $f : A \rightarrow A'$ which is also an
isomorphism of $A_e$-modules. The latter means that $f$ is a
isomorphism such that $f(A_g) = {A'}_g$ for all $g \in G$ and
$f(x) = x$ for any $x \in A_e$.
\end{defi}

\begin{teo}\label{teo_equivalent}
Two quasicrossed products of $G$ over $A_e$ are equivalent if and only
if they determine the same equivalence class of quasicrossed systems
for $G$ over $A_e$.
\end{teo}

\begin{proof}
Consider $A$ and $A'$ two quasicrossed products of $G$ over $A_e$. Let
$(G,A_e, \phi,\sigma,\alpha)$ and $(G,A_e,\phi,\sigma',\alpha')$
be the representatives of the corresponding equivalence classes of
quasicrossed systems for $G$ over $A_e$ and take the systems of units
$\{\overline{g} : g \in G\}$ and $\{\widetilde{g} : g \in G\}$ in
$A$ and $A'$, respectively, which give rise to the above quasicrossed
systems.

First assume that $A'$ and $A$ are equivalent via $f : A'
\rightarrow A$. Because $f( \widetilde g ) \in A_g$ for all $g \in
G$, there is a map $u : G \rightarrow U(A_e)$ with $u(e) = 1$ such
that $f(\widetilde{g}) = u(g)\overline{g}$ for any $g \in G$. We
observe that for given $g \in G$,
\begin{equation*}
\begin{split}
1 = f(1) = f(\widetilde{g}\widetilde{g}_R^{-1}) =
f(\widetilde{g})f(\widetilde{g}_R^{-1}) =
u(g)\overline{g}f(\widetilde{g}_R^{-1}),
\end{split}
\end{equation*}
\noindent so $u(g)$ is an unit in $A_e$ with inverse
\begin{equation*}
\begin{split}
&u(g)^{-1} = \overline{g}f(\widetilde{g}_R^{-1}).
\end{split}
\end{equation*}
\noindent Now we recall the multiplication in $A$ and $A'$,
respectively, to be used in what follows:
\begin{eqnarray*}
&&(x\overline{g})(y\overline{h}) =
x\sigma(g)(y)\alpha(g,h)\overline{gh} \\
&&(x\widetilde{g})(y \widetilde{h}) = x\sigma'(g)(y)\alpha'(g,h)
\widetilde{gh}
\end{eqnarray*}
for any $x,y \in A_e $ and $g,h \in G$. To given  $x \in A_e$ and
$ g \in G$  we have $\widetilde{g}(x\widetilde{e})=
\sigma'(g)(x)\alpha'(g,e) \widetilde{ge}=\sigma'(g)(x)
\widetilde{g }$. Since $f$ is a morphism of algebras we have
\begin{equation*}
\begin{split}
f(\widetilde{g}(x\widetilde{e})) &= f(\widetilde{g})f(
x\widetilde{e} )=f(\widetilde{g})(xf(
\widetilde{e}))=(u(g)\overline{g})(xu(e)\overline{e})=(u(g)\overline{g})(x
\overline{e})\\
&=u(g)\sigma(g)(x)\alpha(g,e)\overline{ge}=u(g)\sigma(g)(x)
\overline{g }
\end{split}
\end{equation*}
and\begin{equation*}
\begin{split}
f(\sigma'(g)(x) \widetilde{g }) &=\sigma'(g)(x)f( \widetilde{g
})=\sigma'(g)(x)u(g)\overline{g}
\end{split}
\end{equation*}
therefore $\sigma'(g)(x)=u(g)\sigma(g)(x)u(g)^{-1}$ proving
\eqref{relation_sigma_sigma'}. Now for $g,h \in G$ we have $
\widetilde{g} \widetilde{h}  =\sigma'(g)(1)\alpha'(g,h)
\widetilde{gh}= \alpha'(g,h) \widetilde{gh}$. Again, since $f$ is
a morphism of algebras,
\begin{equation*}
\begin{split}
f(\widetilde{g} \widetilde{h})&= f(\widetilde{g})f( \widetilde{h})
=(u(g)\overline{g})(u(h)\overline{h})=u(g)\sigma(g)(u(h))\alpha(g,h)\overline{gh}
\end{split}
\end{equation*}
and
\begin{equation*}
\begin{split}
f(\alpha'(g,h) \widetilde{gh})&= \alpha'(g,h) f( \widetilde{gh})=
\alpha'(g,h) u(gh)\overline{gh}
\end{split}
\end{equation*}
\noindent therefore $\alpha'(g,h)=u(g)\sigma(g)(u(h))\alpha(g,h)
u(gh)^{-1}$ getting \eqref{relation_alpha_alpha'}. Thus
$(G,A_e,\sigma,\alpha)$ and $(G,A_e,\sigma',\alpha')$ are
equivalent.

Conversely, suppose that there is a map $u : G \rightarrow U(A_e)$
with $u(e) = 1$ such that \eqref{relation_sigma_sigma'}
and \eqref{relation_alpha_alpha'} are verified. Using again the
multiplication in $A$ and in $A'$, it is easily verified that the
$A_e$-linear extension of the map $f(\widetilde{g}) =
u(g)\overline{g} $ for any $ g \in G$, also denoted by $f$,
provides an equivalence of $A'$ and $A$. In fact, $f$ is an
algebra morphism, because for $ g,h \in G$
\begin{equation*}
\begin{split}
f( \widetilde{g} \widetilde{h} ) = f(\alpha'(g,h)\widetilde{gh}) =
\alpha'(g,h)f(\widetilde{gh}) = \alpha'(g,h)u(gh) \overline{gh}
\end{split}
\end{equation*}
\noindent and
\begin{equation*}
\begin{split}
f(\widetilde{g})f(\widetilde{h})
=(u(g)\overline{g})(u(h)\overline{h})=u(g)\sigma(g)(u(h))\alpha(g,h)\overline{gh}
\end{split}
\end{equation*}
\noindent are equal by \eqref{relation_alpha_alpha'}. Also
satisfies $f(A_g') = A_g$. Indeed, for $x_g \in A_g$, by
Proposition \ref{Prop_page_176}(i) we may write $x_g =
x\overline{g} $ for a certain $ x \in A_e$. Then
\begin{equation*}
\begin{split}
f(xu(g)^{-1}\widetilde{g}) = xu(g)^{-1}f(\widetilde{g}) =
xu(g)^{-1}u(g)\overline{g}=x \overline{g}=x_g,
\end{split}
\end{equation*}
\noindent with $ xu(g)^{-1}\widetilde{g} \in A_g'$. Finally, for
any $x \in A_e$ we have $f(x) =f(x\widetilde{e})=
xf(\widetilde{e}) = x$, completing the proof.
\end{proof}

\begin{defi}
Take an automorphism system $\sigma : G \rightarrow
Aut(\mathbb{K})$, where we consider the field
$\mathbb{K}$ as the associative algebra $B$ on the natural way. A quasicrossed mapping $\delta : G \times G
\rightarrow \mathbb{K}^{\times}$ (see Definition \ref{def_quasicrossed_mapping}) is called a \emph{coboundary} if
there is a function $u : G \rightarrow \mathbb{K}^{\times}$ such
that
\begin{equation*}\label{coboundary_identity}
\begin{split}
\delta(g,h) = u(g)\sigma(g)(u(h))u(gh)^{-1},
\end{split}
\end{equation*}
for any $g,h \in G$.
\end{defi}

\begin{pro}  The quasicrossed systems $ (G,\mathbb{K},\phi,\sigma,\alpha)$ and $ (G,\mathbb{K},\phi,\sigma,\alpha')$ over the associative algebra $\mathbb{K}$ for a fixed cocycle $\phi : G \times G \times G \longrightarrow \mathbb{K}^{\times}$ and an automorphism
system $\sigma : G \rightarrow Aut(B)$  are equivalent if and only
if $\alpha' = \delta \alpha$ for a certain coboundary  $\delta$.
\end{pro}

\begin{proof}
Apply Theorem \ref{teo_equivalent} considering the field
$\mathbb{K}$ as the associative algebra $B$ on the natural way. On
this context, condition \eqref{relation_sigma_sigma'} is trivial
and also from $\mathbb{K}$ is commutative we can rewrite
\eqref{relation_alpha_alpha'} as $\alpha = \delta \alpha'$ where
$\delta$ is a coboundary quasicrossed mapping.
\end{proof}

\section{Cayley (Clifford) quasicrossed products}

Let $A$ be a finite-dimensional (not necessarily associative) algebra with identity element $1$ and an  \textit{involution} $ \varsigma :A\longrightarrow A$, meaning that $ \varsigma $ is an antiautomorphism ($\varsigma(ab)=\varsigma(b)\varsigma(a)$ for all $a,b \in A$) and $\varsigma ^2=id$. Moreover, the involution $ \varsigma$ is called \textit{strong} if it satisfies the property $a+\varsigma(a), a.\varsigma(a) \in \mathbb{K} 1$, for all $a \in A$.
 The Cayley-Dickson process says that we can obtain a new algebra $\overline{A}= A\oplus vA$ of twice the dimension (the elements are denoted by $a,va$, for $a \in A$) with multiplication defined by
\begin{equation*}
\begin{split}
(a+vb) (c+vd) = (ac + \epsilon d\varsigma(b))+  v(\varsigma(a)d+cb),
\end{split}
\end{equation*}
and with a new involution $\overline{\varsigma}$ given by
\begin{equation*}
\begin{split}
\overline{\varsigma}(a+vb) = \varsigma(a) - v b ,
\end{split}
\end{equation*}
for any $a,b,c,d \in A$. The symbol $v$ here is a notation device to label the second copy of $A$ in $\overline{A}$ and $vv=\epsilon 1$, for a fixed nonzero element $\epsilon$ in $ \mathbb{K} $.

\begin{pro}  If $A$ is a  quasicrossed product over the group $ G $ then the algebra $\overline{A}= A\oplus vA$ resulting from the Cayley-Dickson process is a quasicrossed product over the group $\overline{G}=G\times \mathbb{Z}_2$.
\end{pro}

\begin{proof}
First, we note that if  $A = \bigoplus_{g \in G}A_g$ is a $G$-graded algebra, it is easy to see that $\overline{A}= A\oplus vA$ is a $\overline{G}$-graded algebra. We may write the grading $\overline{A}= \bigoplus_{g \in G}A_{(g,0)} \oplus \bigoplus_{g \in G}A_{(g,1)}$, with $ A_{(g,0)}=A_g$ and $ A_{(g,1)}=vA_g$. Now assume that $A = \bigoplus_{g \in G}A_g$ is a  quasicrossed product. For any $g \in G$ exists an unit  $\overline{g}$ in $A_g$, so trivially we have an unit in $ A_{(g,0)} $. Moreover, $v\overline{g}$ is a unity in $ A_{(g,1)}$ with  $\displaystyle v \frac{\varsigma (\overline{g}_R^{-1})}{\epsilon} $ its left inverse and right inverse, as
\begin{equation*}
\begin{split}
& (v \frac{\varsigma (\overline{g}_R^{-1} )}{\epsilon})(v\overline{g})=\epsilon\overline{g}\varsigma (\frac{\varsigma (\overline{g}_R^{-1})}{\epsilon})=\epsilon\overline{g}\frac{1}{\epsilon}\varsigma^2 (\overline{g}_R^{-1}) =
 \overline{g} \varsigma^2 (\overline{g}_R^{-1}) = \overline{g} \overline{g}_R^{-1}  =1,\\
&  (v\overline{g})(v \frac{\varsigma (\overline{g}_R^{-1} )}{\epsilon})=\epsilon\frac{\varsigma (\overline{g}_R^{-1})}{\epsilon} \varsigma ( \overline{g})=  \varsigma (\overline{g}_R^{-1}) \varsigma ( \overline{g})=
 \varsigma (\overline{g} \overline{g}_R^{-1}) =\varsigma (1)  =1
\end{split}
\end{equation*}
completing the proof.
\end{proof}

\noindent In \cite{AM1999}, it was proved that after applying the Cayley-Dickson process to a $\mathbb{K}_FG$ algebra we obtain another $\mathbb{K}_{\overline{F}} \overline{G}$ algebra related to the first one which properties are predictable.

\begin{pro} \label{pro:hbvhh} \cite{AM1999} Let $G$ be a  finite abelian group, $F$ a cochain on it ($\mathbb{K}_FG$ is a $G$-graded
quasialgebra). For any $s:G\longrightarrow \mathbb{K}^{\times}$ with $s(e)=1$ we define $\overline{G}=G\times \mathbb{Z}_2$ and on it the cochain $\overline{F}$ and function $\overline{s}$,
\begin{equation*}
\begin{split}
&\overline{F}(x,y)= F (x,y), \; \overline{F}(x,vy)= s(x) F (x,y),\\
&\overline{F}(vx,y)= F (y,x), \;  \overline{F}(vx,vy)= \epsilon s(x) F (y,x),\\
&\overline{s}(x)= s(x), \;  \overline{s}(vx)= -1, \mbox{ for all } x,y \in G.
\end{split}
\end{equation*}
Here $x\equiv (x,0)$ and $vx\equiv (x,1)$ denote elements of $\overline{G}$, where $ \mathbb{Z}_2=\{0,1 \}$ with operation $1+1=0$. If $\varsigma(x)=s(x)x$ is a strong involution, then $\mathbb{K}_{\overline{F}}\overline{G}$ is the algebra obtained from Cayley-Dickson process applied to $\mathbb{K}_FG$.
\end{pro}

\noindent As the deformed group algebras $\mathbb{K}_FG$ are quasicrossed products, we can improve yet this outcome in the following way. Any deformed group algebra $\mathbb{K}_FG$ (see Example \ref{exa K_FG}), with a
2-cochain  $F$ on $G$, has a natural structure of quasicrossed system considering the associative algebra the field $ \mathbb{K}$ itself. Indeed, consider
in the Definition \ref{def_quasicrossed_mapping} of quasicrossed system, $B = \mathbb{K}$, $ \sigma(g) = id$ and
$\alpha(g,h) = F(g,h)$, for all $g, h \in G$. Conditions \eqref{cond_quasicrossed_system1} and \eqref{cond_quasicrossed_system3} are trivial. From  $\phi(g,h,k) =
\frac{F(g,h)F(gh,k)}{F(h,k)F(g,hk)}$  for any $g, h,k \in G$, meaning that  $\mathbb{K}_FG$
is a coboundary graded quasialgebra, we obtain assertion \eqref{cond_quasicrossed_system2}. Reciprocal, it is easy to see that any quasicrossed system $(G,\mathbb{K},\phi,\sigma,\alpha)$, where $\mathbb{K}$ is a field,  corresponds to a deformed group algebra.

\begin{pro} Assume that the deformed group algebra $\mathbb{K}_FG$ held as a quasicrossed product, corresponds to a quasicrossed system $(G,\mathbb{K},\phi,\sigma,\alpha)$ for $G$ over $\mathbb{K}$. If the  algebra $\mathbb{K}_{\overline{F}}\overline{G}$ obtained from the Cayley-Dickson process applied to $\mathbb{K}_FG$  corresponds to a quasicrossed system
$(\overline{G},\mathbb{K},\overline{\phi},\overline{\sigma},\overline{\alpha})$ for $\overline{G}$ over $\mathbb{K}$ then
\begin{equation*}
\begin{split}
&\overline{\alpha}(x,y)=\alpha (x,y),  \;  \overline{\alpha}(x,vy)= s(x) \alpha (x,y),\\
&\overline{\alpha}(vx,y)= \alpha (y,x), \;  \overline{\alpha}(vx,vy)= \epsilon s(x) \alpha (y,x),  \mbox{ for all } x,y \in G.
\end{split}
\end{equation*}
\end{pro}

\begin{proof} Applying Proposition \ref{pro:hbvhh} we have successively,
\begin{equation*}
\begin{split}
&\overline{\alpha}(x,y)= \overline{F}(x,y)= F (x,y) = \alpha (x,y), \\
&\overline{\alpha}(x,vy)=\overline{F}(x,vy)= s(x) F (x,y)= s(x) \alpha (x,y),\\
&\overline{\alpha}(vx,y)=\overline{F}(vx,y)= F (y,x)= \alpha (y,x), \\
 &\overline{\alpha}(vx,vy)=\overline{F}(vx,vy)= \epsilon s(x) F (y,x) = \epsilon s(x) \alpha (y,x),  \mbox{ for all } x,y \in G.
\end{split}
\end{equation*}

\end{proof}

\section{Simple quasicrossed products}

The aim of this section is to study simple quasicrossed products. We recall the notion of simple graded quasialgebra in general.

\begin{defi}
A graded quasialgebra $A$ is \emph{simple} if $\displaystyle {A}^2 \neq \{0\}$ and it has no
proper graded ideals, or equivalently, if the ideal generated by
each nonzero homogeneous element is the whole quasialgebra.
\end{defi}
\noindent To study simple quasicrossed product we introduce the definition of representation of a graded quasialgebra.
In the following definition of modules, $A =\bigoplus_{g
\in G}A_g$ is a $G$-graded quasialgebra with structure given by $\phi$ and $V =\bigoplus_{k \in
G}V_k$ is a graded vector space over the same group $G$.
We denote by $\bullet $ the product defined in $A$. First we emphasize that  the quasiassociative law in $A$  is performed by $\bullet \circ
(\bullet \otimes id ) = \bullet \circ (id \otimes \bullet) \circ
\Phi_{A,A,A}$ and can be represented by the
following commutative diagram\\
\xymatrix{ \hspace*{2cm}
 &A \otimes A\otimes A\ar[r]^{\Phi_{A,A,A}}  \ar[d]_{\bullet\otimes id}  &
A\otimes A\otimes A\ar[r]^{ \;\; id\otimes \bullet} & A\otimes A \ar[d]^{\bullet}\\
&A\otimes A\ar[rr]_{\bullet}&&A }

\begin{defi}
Consider a degree-preserving map $\varphi : A \otimes V
\longrightarrow V$ and denote $x_g.v_k :=\varphi(x_g,v_k) $. If
for homogeneous elements $x_g \in A_g, x_h \in A_h, v_k \in V_k$,
\begin{equation*}
\begin{split}
&(x_g x_h).v_k = \phi(g,h,k)x_g.(x_h.v_k),\\
& 1.v_k =  v_k,
\end{split}
\end{equation*}
then $V$ is called a \textit{left graded module} over $A$ (or a
\textit{left} $A$-\textit{graded-module}).
\end{defi}
\noindent The condition of left graded module is a natural
polarization of the quasiassociativity of the product on $A$, as
we can see by the following commutative diagram\\
\xymatrix{ \hspace*{2cm}
 &A\otimes A\otimes V\ar[r]^{\Phi_{A,A,V}}  \ar[d]_{\bullet\otimes id}  &
A\otimes A\otimes V\ar[r]^{ \;\; id\otimes \varphi} & A\otimes V
\ar[d]^{\varphi}\\
&A\otimes V\ar[rr]_{\varphi}&&V }

\begin{defi}
Consider a degree-preserving map $\psi : V \otimes A
\longrightarrow V$ and denote $v_k.x_g:=\psi(v_k,x_g)$. If for
homogeneous elements $x_g \in A_g, x_h \in A_h, v_k \in V_k$,
\begin{equation*}
\begin{split}
&(v_k.x_g).x_h = \phi(k,g,h)v_k.(x_g x_h),\\
&  v_k.1 = v_k,
\end{split}
\end{equation*}
then $V$ is called a \textit{right graded module} over $A$ (or a
\textit{right} $A$-\textit{graded-module}).
\end{defi}

\noindent Similarly, the condition of right graded module is
represented in the following commutative diagram\\
\xymatrix{ \hspace*{2cm}
 &V\otimes A\otimes A\ar[r]^{\Phi_{V,A,A}}  \ar[d]_{\psi\otimes id}  &
V\otimes A\otimes A\ar[r]^{ \;\; id\otimes \bullet} & V\otimes A
\ar[d]^{\psi}\\
&V\otimes A\ar[rr]_{\psi}&&V }

\begin{defi} If $V$ is a
left and right graded module of $A$ and if for homogeneous
elements $x_g \in A_g, x_h \in A_h, v_k \in V_k$,
\begin{equation*}
\begin{split}
&(x_g.v_k).x_h = \phi(g,k,h)x_g.(v_k.x_h),
\end{split}
\end{equation*}
then $V$ is called a \textit{graded bimodule} over $A$ (or an
$A$-\textit{graded-bimodule}).
\end{defi}

\noindent Moreover, the condition of graded bimodule is
represented by the following commutative diagram\\
\xymatrix{ \hspace*{2cm}
 &A\otimes V\otimes A\ar[r]^{\Phi_{A,V,A}}  \ar[d]_{\varphi\otimes id}
& A\otimes V\otimes A\ar[r]^{ \;\; id\otimes \psi} & A\otimes V
\ar[d]^{\varphi}\\
&V\otimes A\ar[rr]_{\psi}&&V }

\noindent Now we present some examples of graded modules over graded quasialgebras.
\begin{exa}
Consider the antiassociative quasialgebra
$A:=\widetilde{Mat}_{2,2}(\mathbb{K})$ of the square matrices over the field $\mathbb{K}$ graded by the group  $\mathbb{Z}_2$ such that
$A_{\bar 0}:= \langle E_{11}, E_{22} \rangle $ and $A_{\bar 1}:= \langle E_{12}, E_{21} \rangle $
satisfying the multiplication
\begin{equation*}
\begin{split}
\left(%
\begin{array}{cc}
 a_1 & v_1 \\
 w_1 & b_1 \\
\end{array}%
\right) \cdot \left(%
\begin{array}{cc}
 a_2 & v_2 \\
 w_2 & b_2 \\
\end{array}%
\right) = \left(%
\begin{array}{cc}
 a_1a_2+v_1w_2 & a_1v_2+v_1b_2 \\
 w_1a_2+b_1w_2 & -w_1v_2+b_1b_2 \\
\end{array}%
\right) .
\end{split}
\end{equation*}
Consider the vector space $M:= \langle m, n \rangle$ endowed with the grading by the group $\mathbb{Z}_2$ with
$M_{\bar 0}:= \langle m \rangle $ and $M_{\bar 1}:= \langle n \rangle $ acting on $A$ as follows:
\begin{equation*}
\begin{split}
& mE_{11}=nE_{21}=m, \; mE_{12}=nE_{22}=n,\\
& mE_{22}=mE_{21}=nE_{11}=nE_{12}=0;
\end{split}
\end{equation*}
and on the other side,
\begin{equation*}
\begin{split}
& E_{22}m=E_{21}n=m, \; -E_{12}m=E_{11}n=n,\\
& E_{11}m=E_{21}m=E_{22}n=E_{12}n=0.
\end{split}
\end{equation*}
We check easily  that $M$ is both a right $A$-graded-module and a left $A$-graded-module, although the two structures are not compatible, that is, $M$ is not a graded bimodule over $A$ (just note that $(E_{21}n)E_{12}=n$ and $E_{21}(nE_{12})=0$).
\end{exa}

\begin{exa}
We consider a commutative $G$-graded quasialgebra $\mathbb{K}_FG$.
Applying Proposition 4.6 in \cite{AM1999}, we know that the quasialgebra obtained from  $\mathbb{K}_FG$ by the Cayley-Dickson doubling process can be defined by the same cocycle graded by $G$ being  the degree of the element $vx$ equal to the degree of  $x$, for $x \in G$. Then the subspace $v\mathbb{K}_FG$ constitutes an example of a graded bimodule over $\mathbb{K}_FG$.

Take the concrete case of the simple commutative non-associative algebra $\mathbb{K}_F\mathbb{Z}_3$ with multiplication defined by the table (cf. \cite[Proposition 5.3]{AM1999}):
\begin{center}
\begin{tabular}{c|ccc}
  &   $e$ & $e_1$ & $e_2$ \\
      \hline
  $e$ & $e$ & $e_1$ & $e_2$ \\
  $e_1$ & $e_1$ & $-e_2$ & $e$ \\
  $e_2$ & $e_2$ & $e$ & $e_1$
\end{tabular}
\end{center}
\smallskip
and diagonal involution given by $s(e)=1$ and $s(e_1)=s(e_2)=-1$. Then we can obtain a $\mathbb{Z}_3$-graded quasialgebra with the same cocycle graded by $\mathbb{Z}_3$, constructed by the Cayley-Dickson process  from  $\mathbb{K}_F\mathbb{Z}_3$, where $M:=v\mathbb{K}_F\mathbb{Z}_3$ is a graded bimodule over $\mathbb{K}_F\mathbb{Z}_3$ define
by $M:= \langle ve,ve_1,ve_2 \rangle$ endowed with the grading by the group $\mathbb{Z}_3$ with
$M_{\bar 0}:= \langle ve \rangle $,  $M_{\bar 1}:= \langle ve_1 \rangle $ and $M_{\bar 2}:= \langle ve_2 \rangle $ acting on $\mathbb{K}_F\mathbb{Z}_3$ on the right   and on the left as follows:
\begin{center}
\[
\hspace*{-3cm}
\vbox{\offinterlineskip
\begin{tabular}{c|ccc}
  &   $e$ & $e_1$ & $e_2$ \\
      \hline
  $ve$ & $ve$ & $ve_1$ & $ve_2$ \\
  $ve_1$ & $ve_1$ & $-ve_2$ & $ve$ \\
  $ve_2$ & $ve_2$ & $ve$ & $ve_1$
\end{tabular}}
\hspace*{-7cm}
\vbox{\offinterlineskip
\begin{tabular}{ccc|c}
     $ e$ & $ e_1$ & $ e_2$ &\\
      \hline
  $ve$ & $-ve_1$ & $-ve_2$ &  $ve$  \\
 $ve_1$ & $ve_2$ & $-ve$ &  $ve_1$ \\
  $ve_2$ & $-ve$ & $-ve_1$ &  $ve_2$
\end{tabular}}
\]
\end{center}
\smallskip
\end{exa}

\begin{defi}
Let $V$ be an $A$-graded-bimodule, a \textit{graded submodule} $W
\subset V$ is a submodule (meaning $AW \subset W$) such that $W =
\oplus_{g \in G}(W \cap V_g)$. We said that a $A$-graded-bimodule $V$ is \textit{simple} if it contains no proper graded submodules.
\end{defi}

\begin{exa}\label{example_bimodules}
A graded quasialgebra $A$ is an $A$-graded-bimodule acting on
itself by the product map. Also, each one $A_g \subset A$ is an
$A_e$-graded-bimodule and a graded submodule of $A$, for any $g
\in G$.
\end{exa}

\begin{defi}
Consider two $A$-graded-bimodules $V$ and $V'$. An $A$-linear $f :
V \rightarrow V'$ is said to be a \textit{graded morphism of
degree} $g$ if $f(V_h) \subset V_{hg}'$, for all $h \in G$.
\end{defi}


\noindent Now we recall the definition of radical of a graded quasialgebra.

\begin{defi}
 Let $A$ be a graded quasialgebra. The radical of $A$ is defined by the intersection
\begin{equation*}
\begin{split}
\mbox{rad}(A)=\cap \{ \mbox{ann } M: M \mbox{ simple graded left } A\mbox{-module} \},
\end{split}
\end{equation*}
where ann $M$ is the annihilator of $M$ in $A$.
\end{defi}
\noindent The radical of a graded quasialgebra $A$ is a proper graded ideal of $A$. So $\mbox{rad}(A)=\{0\}$ if $A$ is simple.
\begin{teo}
Let $A$ be a simple quasicrossed product. Then $A_e$ is a semisimple associative algebra.
\end{teo}

\begin{proof}
It is similar to the proof of Theorem 4.3 in \cite{AS2006}. Let $\mbox{J}(A_e)$ denote the Jacobson radical of the associative algebra $A_e$. Given a simple graded $A$-module $M =\bigoplus _{g \in G}M_g$, each $M_g$ is a simple $A_e$-module. Thus if $a_0 \in \mbox{J}(A_e)$ then $a_0M_g=0$, $\forall {g \in G}$. Therefore $\mbox{J}(A_e)\subseteq \mbox{rad}(A)=\{0\}$ and $A_e$ is semisimple.
\end{proof}

\noindent In case $\displaystyle G=\mathbb{Z}_2$, the classification of quasialgebras that have semisimple associative null part was done in \cite{AEPI2002}, so we have the following result.

\begin{teo}
Any simple quasicrossed product ${A}$ of $\mathbb{Z}_2$ over ${A}_{\bar 0}$ is isomorphic to one of the following algebras:
\begin{enumerate}
  \item[(i)]  $\displaystyle \mbox{Mat}_n(\Delta)$, for some $n$ and some division antiassociative quasialgebra $\Delta$;
  \item[(ii)]  $\widetilde{Mat}_{n,m}(D)$, for some natural numbers $n$ and $m$ and some division algebra $D$.
\end{enumerate}
Moreover, the natural numbers $n$ and $m$ are uniquely determined by $A$ and so are (up to isomorphism) the division antiassociative quasialgebra $\Delta$ and the division algebra $D$.
\end{teo}

\section{The $\mathbb{K}_FG$ case}

\begin{defi}\label{defi_center}
The \emph{centralizer} of a subset $X$ of a graded quasialgebra
$A$ is the subset

\begin{equation*}
\begin{split}
C_A(X) := \{a \in A : ax = xa, \hspace{0.2cm} \hbox{for all}
\hspace{0.2cm} x \in X\}
\end{split}
\end{equation*}
\noindent In particular, if $X = A$ then $C_A(A) = C(A)$ is the
\emph{center} of the graded quasialgebra.
A graded quasialgebra $A$ over a field $\mathbb{K}$ is called
\emph{central} if $C(A) = \mathbb{K}$.
\end{defi}

\begin{defi}\label{defi_central_simple}
A finite-dimensional graded quasialgebra $A$ over a field
$\mathbb{K}$ is \emph{central simple} if $A$ is simple and
central.
In particular, if $A$ is central the concept of central simple
agree with the usual simplicity.
\end{defi}

\begin{exa}
\begin{enumerate}
\item Any simple algebra is a central simple graded quasialgebra
over its center.

\item The complex number $\mathbb{C}$ is a central simple algebra over
$\mathbb{C}$, but not over the real numbers $\mathbb{R}$ (the center
of $\mathbb{C}$ is all  $\mathbb{C}$, not just $\mathbb{R}$).

\item The quaternions $\mathbb{H}$ form a 4-dimensional central
simple algebra over $\mathbb{R}$.
\end{enumerate}
\end{exa}

\begin{teo}
The deformed group algebra $A $ corresponding to a given quasicrossed system $ (G,\mathbb{K},\phi,\sigma,\alpha)$ is a
central simple quasialgebra over the field $\mathbb{K}$.
\label{teo:g89p0oi}
\end{teo}

\begin{proof}
If an element $x = \sum\limits_{g \in G}k_g\overline{g} \in A$,
with $k_g \in \mathbb{K}$, belongs to the center of $A$, then
it commutes with all elements $k$ of $\mathbb{K}$, and therefore

\begin{equation*}
\begin{split}
\sum\limits_{g \in G}(kk_g)\overline{g} =
(k\overline{e})\Bigl(\sum\limits_{g \in G}k_g\overline{g}\Bigr) =
\Bigl(\sum\limits_{g \in G}k_g\overline{g}\Bigl)(k\overline{e}) =
\sum\limits_{g \in G}k_g\sigma(g)(k)\alpha(g,e)\overline{ge}
\end{split}
\end{equation*}

\begin{equation*}
\begin{split}
= \sum\limits_{g \in G}k_g\sigma(g)(k)\overline{g}
\end{split}
\end{equation*}

\noindent This means that whenever $k_g \neq 0$, then $k = \sigma(g)(k)$ for
all $k \in \mathbb{K}$, and consequently $\sigma (g) = id$. Thus
$C_A(\mathbb{K}) = \mathbb{K}$, and so $C(A) \subset \mathbb{K}$. However, if   $k \in \mathbb{K}
\cap C(A)$, then $k\overline{g} = \overline{g}k =
\sigma(g)(k)\overline{g}$, for all $g \in G$, and consequently $k
\in \{a \in A : a\overline{g} = \sigma(g)(a)\overline{g},
\hspace{0.2cm} \mbox{for any } g \in G\} = \mathbb{K}$. Thus $C(A)
= \mathbb{K}$ and the quasialgebra $A$ is central.

Now, let $I$ be a graded ideal of $A$. Choose a non-zero element
$x = \sum_{g \in G}x_g\overline{g}$ in $ I$. Multiplying $x$ by
$\overline{g}$, for a suitable $g$, we can assume that $x_e \neq
0$. Let $k$ be an arbitrary element of $\mathbb{K}$. Since $A$ is central then $kx - xk=0$. But

\begin{equation*}
\begin{split}
kx - xk = \sum\limits_{g \in G}kx_g\overline{g} - \sum\limits_{g
\in G}x_g\overline{g}k = \sum\limits_{g \in G}kx_g\overline{g} -
\sum\limits_{g \in G}x_g\sigma(g)(k)\overline{g} = \sum\limits_{g \in
G}(kx_g - \sigma(g)(k)x_g)\overline{g}.
\end{split}
\end{equation*}

\noindent  For each $g \neq e$, we have that $\sigma(g) $ is not the identity map in $\mathbb{K}$ and by the previous calculations $ x_g = 0$. So  $x = x_e\overline{e}$ is an invertible element of the ideal $I$, hence $I = A$. Therefore $A$ is simple as desired.
\end{proof}

\end{document}